\documentclass[reqno]{amsart}

\usepackage{graphicx}
\usepackage{amscd}
\usepackage{amsmath, amssymb}
\usepackage{graphics}
\usepackage{graphicx}
\usepackage{epsfig}
\usepackage{xcolor}
\usepackage{soul}
\usepackage{tikz}

\newtheorem{theorem}{Theorem}

\newtheorem{definition}[theorem]{Definition}

\begin{document}

\title[Local representations of $FVB_n$]{Local Representations of the Flat Virtual Braid Group}

\author{Mohamad N. Nasser, Mohammad Y. Chreif and Malak M. Dally}
\address{Mohamad N. Nasser\\
         Department of Mathematics and Computer Science\\
         Beirut Arab University\\
         P.O. Box 11-5020, Beirut, Lebanon}
\email{m.nasser@bau.edu.lb}
\address{Mohamad Y. Chreif\\
         Department of Mathematics and Physics\\
         Lebanese International University\\
         Rayak, Lebanon}
\email{mohamad.chreif@liu.edu.lb}
\address{Malak M. Dally\\
         Department of Business Economics\\
         Lebanese International University\\
         Rayak, Lebanon}
\email{malak.dally@liu.edu.lb}

\begin{abstract}
We prove that any complex local representation of the flat virtual braid group, $FVB_2$, into $\mathrm{GL}_2(\mathbb{C})$, has one of the types $\lambda_i: FVB_2 \rightarrow \mathrm{GL}_2(\mathbb{C})$, $1\leq i\leq 12$. We find necessary and sufficient conditions that guarantee the irreducibility of representations of type $\lambda_i$, $1\leq i\leq 5$, and we prove that representations of type $\lambda_i$, $6\leq i\leq 12$,  are reducible. Regarding faithfulness, we find necessary and sufficient conditions for representations of type $\lambda_6$ or $\lambda_7$ to be faithful. Moreover, we give sufficient conditions for representations of type $\lambda_1$, $\lambda_2$, or $\lambda_4$ to be unfaithful, and we show that representations of type $\lambda_i$, $i=3, 5, 8, 9, 10, 11, 12$ are unfaithful. We prove that any complex homogeneous local representations of the flat virtual braid group, $FVB_n$, into  $\mathrm{GL}_{n}(\mathbb{C})$, for $n\geq 3$, has one of the types $\gamma_i: FVB_n \rightarrow \mathrm{GL}_n(\mathbb{C})$, $i=1, 2$. We then prove that representations of type $\gamma_1: FVB_n \rightarrow \mathrm{GL}_n(\mathbb{C})$ are reducible for $n\geq 6$, while representations of type $\gamma_2: FVB_n \rightarrow \mathrm{GL}_n(\mathbb{C})$ are irreducible if and only if $b\neq y$, for $n\geq 3$. Then, we show that representations of type $\gamma_1$ are unfaithful for $n\geq 3$ and that representations of type $\gamma_2$ are unfaithful if $y=b$. Furthermore, we prove that any complex homogeneous local representation of the flat virtual braid group, $FVB_n$, into $\mathrm{GL}_{n+1}(\mathbb{C})$, for all $n\geq 4$, has one of the types $\delta_i: FVB_n \rightarrow \mathrm{GL}_{n+1}(\mathbb{C})$, $1\leq i\leq 8$. We prove that these representations are reducible for $n\geq 10$. Then, we show that representations of types $\delta_i$, $i\neq 5, 6$, are unfaithful, while representations of types $\delta_5$ or $\delta_6$ are unfaithful if $x=y$.
\end{abstract}

\maketitle

\renewcommand{\thefootnote}{}

\vspace*{-0.15cm}

\section{Introduction}

\vspace*{0.1cm}

The braid group on $n$ strands, $B_n$, is the abstract group introduced by E. Artin \cite{E.A}, in 1926, generated by $\sigma_1,\sigma_2, \ldots,\sigma_{n-1}$. Braid group representations have several applications in different domains such as quantum computing, knot theory and many others \cite{DRW}. Consequently, researchers are interested in defining braid group representations and in studying their properties, such as irreducibility and faithfulness. A group $G$ is said to be linear if a faithful representation can be defined on $G$. One of the most intriguing questions was whether the braid group is linear or not. The first representation used to answer the question of linearity of the braid group was the Burau representation defined in \cite{1}. Burau representation has been proved to be faithful for $n\leq 3$ in \cite{1234} and unfaithful for $n\geq 5$ in \cite{Moo}, \cite{long} and \cite{5678}; whereas the case $n=4$ remains open. Later on, it was shown in \cite{Big2001} and \cite{Kram2002} that Lawrence-Krammer-Bigelow representation of $B_n$ defined in \cite{Law90} is faithful for all $n\geq 2$. Thus, $B_n$ is shown to be linear due to this representation. 

\vspace*{0.1cm}

In 1999, L. Kauffman \cite{123, 888} introduced the virtual braids together with virtual knots and links, where he showed that virtual knots and links can be represented by the closure of virtual braids. The virtual braid group of $n$ strands, $VB_n$, is a group extension of the braid group, $B_n$, with the generators $\sigma_1,\sigma_2, \ldots,\sigma_{n-1}$ of $B_n$ together with the additional generators $\rho_1,\rho_2, \ldots,\rho_{n-1}$. Determining virtual braid group representations, particularly finding extensions of braid group representations to the virtual braid group, is of great interest for researchers nowadays \cite{BN, BNA}.

\vspace*{0.1cm}

Kauffman introduced the flat virtual braid group, $FVB_n$, as a quotient of the group $VB_n$ by the relations $\sigma_i^2=1$. Defining representations and studying their properties is a crucial and important question for researchers. This is because group representations provide researchers with information on algebraic structure about a given group. In 2023, V. Bardakov et al. defined, in \cite{BC}, flat virtual braid group representations which do not preserve the forbidden relations. Also, in 2023, B. Chuzhinov and A. Vesnin \cite{CV} defined representations of the flat virtual braid group by automorphisms of free group.

\vspace*{0.1cm}

In 2013, Y. Mikhalchishina \cite{17} determined all homogeneous local representations of the braid group, $B_n$, into $\mathrm{GL}_n(\mathbb{C})$, for all $n \geq 3$ . Then, in 2025, T. Mayassi and M. Nasser \cite{37} determined all homogeneous local representations of $B_n$ into $\mathrm{GL}_{n+1}(\mathbb{C})$, for all $n \geq 4$. In 2024, M. Chreif and M. Dally \cite{20} studied the irreducibility of the homogeneous local representations defined by Mikhalchishina.  Also, in 2025, Nasser \cite{40} considered several homogeneous local representations of $B_n$ and found their extensions to the singular braid monoid $SM_n$ and the singular braid group $SB_n$. 

\vspace*{0.1cm}

In section 3, we prove that any representation of $FVB_2$, into  $\mathrm{GL}_2(\mathbb{C})$, has one of the twelve types $\lambda_i: FVB_2 \rightarrow \mathrm{GL}_2(\mathbb{C})$, $1\leq i\leq 12$. We study the irreducibility of these representations. We find necessary and sufficient conditions that guarantee the irreducibility of representations of type $\lambda_i$, $1\leq i\leq 5$, and we prove that representations of type $\lambda_i$, $6\leq i\leq 12$,  are reducible. Regarding faithfulness, we prove that representations of type $\lambda_i,$ $i=3,5,8,9,10,11,12,$ are unfaithful, while we find sufficient conditions for representations of type $\lambda_i,$ $i=1,2,4,$ to be unfaithful. In addition, we give necessary and sufficient conditions for representations of type $\lambda_i, i=6,7,$ to be faithful.
\\
In section 4, we prove that any homogeneous local representation of $FVB_n$, into $\mathrm{GL}_n(\mathbb{C})$, for $n\geq 3$, has one of the two types $\gamma_i: FVB_n \rightarrow \mathrm{GL}_n(\mathbb{C})$, $i=1, 2$. We prove that any representation of type $\gamma_1: FVB_n \rightarrow \mathrm{GL}_n(\mathbb{C})$ is reducible for $n\geq 6$, while any representation of type $\gamma_2: FVB_n \rightarrow \mathrm{GL}_n(\mathbb{C})$ is irreducible if and only if $y\neq b$, for $n\geq 3$.  Moreover, we prove that representations of type $\gamma_1$ are unfaithful while representations of type $\gamma_2$ are unfaithful if $y=b$. 
\\
In section 5, we prove that any complex homogeneous local representation of the flat virtual braid group, $FVB_n$, into $\mathrm{GL}_{n+1}(\mathbb{C})$, for all $n\geq 4$, has one of the eight types $\delta_i: FVB_n \rightarrow \mathrm{GL}_{n+1}(\mathbb{C})$, $1\leq i\leq 8$. We prove that these representations are reducible for $n\geq 10$. Also, we prove, for $n\geq 3$, that representations of type $\delta_i,$ $i=1,2,3,4,5,7,8,$ are unfaithful and that representations of type $\delta_i,$ $i=6,7,$ are unfaithful if $x=y$.

\vspace*{0.1cm}

\section{Preliminaries}

\vspace*{0.1cm}

\begin{definition}
\cite{E.A} The braid group, $B_n$, is the group defined by the generators $\sigma_1,\sigma_2,\ldots,\sigma_{n-1}$ with the relations
\begin{equation} \label{eqs1}
\ \ \ \ \sigma_i\sigma_{i+1}\sigma_i = \sigma_{i+1}\sigma_i\sigma_{i+1} ,\hspace{0.45cm} i=1,2,\ldots,n-2,
\end{equation}
\begin{equation} \label{eqs2}
\sigma_i\sigma_j = \sigma_j\sigma_i , \hspace{1.48cm} |i-j|\geq 2.
\end{equation}
\end{definition}

\begin{definition}
\cite{123} The virtual braid group, $VB_n$, is the group defined by the generators $\sigma_1,\sigma_2, \ldots,\sigma_{n-1}$ and $\rho_1,\rho_2, \ldots, \rho_{n-1}$. In addition to the relations \eqref{eqs1} and \eqref{eqs2}, the generators $\sigma_i$ and $\rho_i$ of $VB_n$ satisfy the following relations:
\begin{equation} \label{eqs3}
\ \ \ \ \rho_i\rho_{i+1}\rho_i = \rho_{i+1}\rho_i\rho_{i+1}, \hspace{0.55cm} i=1,2,\ldots,n-2,
\end{equation}
\begin{equation} \label{eqs4}
\rho_i\rho_j = \rho_j\rho_i ,\hspace{1.55cm} |i-j|\geq 2,
\end{equation}
\begin{equation} \label{eqs5}
\ \ \ \ \ \ \ \ \ \ \ \ \rho_i^2 = 1 ,\hspace{2cm} i=1,2,\ldots,n-1,
\end{equation}
\begin{equation} \label{eqs6}
\sigma_i\rho_j=\rho_j\sigma_i ,\hspace{1.6cm} |i-j|\geq 2,
\end{equation}
\begin{equation} \label{eqs7}
\ \ \ \ \rho_i\rho_{i+1}\sigma_i=\sigma_{i+1}\rho_i\rho_{i+1}, \hspace{0.55cm} i=1,2,\ldots,n-2.
\end{equation}    
\end{definition}

\begin{definition}
\cite{12345} The flat virtual braid group, $FVB_n$, is the group defined by the generators $\sigma_1,\sigma_2, \ldots,\sigma_{n-1}$ and $\rho_1,\rho_2, \ldots, \rho_{n-1}$, satisfying the relations of the virtual braid group, $VB_n$, together with the relations
\begin{equation} \label{eqs8}
\ \ \ \ \ \ \ \ \ \ \ \ \sigma_i^2 = 1 ,\hspace{2cm} i=1,2,\ldots,n-1.
\end{equation}    
\end{definition}

\begin{definition}
Let $G$ be a group with generators $g_1,g_2,\ldots,g_{n-1}$. A representation $\theta: G \rightarrow \mathrm{GL}_{m}(\mathbb{Z}[t^{\pm 1}])$ is said to be local if it is of the form
$$\theta(g_i) =\left( \begin{array}{c|@{}c|c@{}}
   \begin{matrix}
     I_{i-1} 
   \end{matrix} 
      & 0 & 0 \\
      \hline
    0 &\hspace{0.2cm} \begin{matrix}
   		M_i
   		\end{matrix}  & 0  \\
\hline
0 & 0 & I_{n-i-1}
\end{array} \right) \hspace*{0.2cm} \text{for} \hspace*{0.2cm} 1\leq i\leq n-1,$$ 
where $M_i \in \mathrm{GL}_k(\mathbb{Z}[t^{\pm 1}])$ with $k=m-n+2$ and $I_r$ is the $r\times r$ identity matrix. The local representation is said to be homogeneous if all the matrices $M_i$ are equal.
\end{definition}

\begin{definition}
Let $H$ be a group with $2(n-1)$ generators $g_1,g_2,\ldots,g_{n-1}$ and $h_1,h_2,\ldots,h_{n-1}$. A local representation $\theta: H \rightarrow \mathrm{GL}_{m}(\mathbb{Z}[t^{\pm 1}])$ is a representation of the form
$$\theta(g_i) =\left( \begin{array}{c|@{}c|c@{}}
   \begin{matrix}
     I_{i-1} 
   \end{matrix} 
      & 0 & 0 \\
      \hline
    0 &\hspace{0.2cm} \begin{matrix}
   		M_i
   		\end{matrix}  & 0  \\
\hline
0 & 0 & I_{n-i-1}
\end{array} \right) \text{ and }\ \theta(h_i) =\left( \begin{array}{c|@{}c|c@{}}
   \begin{matrix}
     I_{i-1} 
   \end{matrix} 
      & 0 & 0 \\
      \hline
    0 &\hspace{0.2cm} \begin{matrix}
   		N_i
   		\end{matrix}  & 0  \\
\hline
0 & 0 & I_{n-i-1}
\end{array} \right) $$
for $1\leq i\leq n-1,$ where $M_i,N_i \in \mathrm{GL}_k(\mathbb{Z}[t^{\pm 1}])$ with $k=m-n+2$ and $I_r$ is the $r\times r$ identity matrix. In this case, $\theta$ is homogeneous if all the matrices $M_i$ are equal and all the matrices $N_i$ are equal.   
\end{definition}

We list the representations below as examples of the homogeneous local representations of the braid group $B_n$.
 
\begin{definition} \cite{1} \label{defBurau}
The Burau representation $\Psi_B: B_n\rightarrow \mathrm{GL}_n(\mathbb{Z}[t^{\pm 1}])$ is the representation defined by
$$\sigma_i\rightarrow \left( \begin{array}{c|@{}c|c@{}}
   \begin{matrix}
     I_{i-1} 
   \end{matrix} 
      & 0 & 0 \\
      \hline
    0 &\hspace{0.2cm} \begin{matrix}
   	1-t & t\\
   	1 & 0\\
\end{matrix}  & 0  \\
\hline
0 & 0 & I_{n-i-1}
\end{array} \right) \hspace*{0.2cm} \text{for} \hspace*{0.2cm} 1\leq i\leq n-1.$$ 
\end{definition}

\begin{definition} \cite{19} \label{Fdef}
The $F$-representation $\Psi_F: B_n \rightarrow \mathrm{GL}_{n+1}(\mathbb{Z}[t^{\pm 1}])$ is the representation defined by
$$\sigma_i\rightarrow \left( \begin{array}{c|@{}c|c@{}}
   \begin{matrix}
     I_{i-1} 
   \end{matrix} 
      & 0 & 0 \\
      \hline
    0 &\hspace{0.2cm} \begin{matrix}
   		1 & 1 & 0 \\
   		0 &  -t & 0 \\   		
   		0 &  t & 1 \\
   		\end{matrix}  & 0  \\
\hline
0 & 0 & I_{n-i-1}
\end{array} \right) \hspace*{0.2cm} \text{for} \hspace*{0.2cm} 1\leq i\leq n-1.$$ 
\end{definition}

\vspace*{0.1cm}

For more information regarding the irreducibility of Burau representation and the $F$-representation, see \cite{75} and \cite{55}.

\vspace*{0.1cm}

\begin{theorem} \cite{17} \label{ThmM}
Consider $n\geq 3$ and let $\beta: B_n \rightarrow \mathrm{GL}_n(\mathbb{C})$ be a non-trivial homogeneous local representation of $B_n$. Then, $\beta$ has one of the following types.
\begin{itemize}
\item[(1)] $\beta_1: B_n \rightarrow \mathrm{GL}_n(\mathbb{C}) \hspace*{0.15cm} \text{such that } 
\beta_1(\sigma_i) =\left( \begin{array}{c|@{}c|c@{}}
   \begin{matrix}
     I_{i-1} 
   \end{matrix} 
      & 0 & 0 \\
      \hline
    0 &\hspace{0.2cm} \begin{matrix}
   		a & \frac{1-a}{c}\\
   		c & 0
   		\end{matrix}  & 0  \\
\hline
0 & 0 & I_{n-i-1}
\end{array} \right), \\ \text{ where } c \neq 0, a\neq 1 \hspace*{0.15cm} \text{for} \hspace*{0.2cm} 1\leq i\leq n-1.$
\item[(2)] $\beta_2: B_n \rightarrow \mathrm{GL}_n(\mathbb{C}) \hspace*{0.15cm} \text{such that }
\beta_2(\sigma_i) =\left( \begin{array}{c|@{}c|c@{}}
   \begin{matrix}
     I_{i-1} 
   \end{matrix} 
      & 0 & 0 \\
      \hline
    0 &\hspace{0.2cm} \begin{matrix}
   		0 & \frac{1-d}{c}\\
   		c & d
   		\end{matrix}  & 0  \\
\hline
0 & 0 & I_{n-i-1}
\end{array} \right),\\ \text{ where } c\neq 0, d\neq 1 \hspace*{0.2cm} \text{for} \hspace*{0.2cm} 1\leq i\leq n-1.$
\item[(3)] $\beta_3: B_n \rightarrow \mathrm{GL}_n(\mathbb{C}) \hspace*{0.15cm} \text{such that } 
\beta_3(\sigma_i) =\left( \begin{array}{c|@{}c|c@{}}
   \begin{matrix}
     I_{i-1} 
   \end{matrix} 
      & 0 & 0 \\
      \hline
    0 &\hspace{0.2cm} \begin{matrix}
   		0 & b\\
   		c & 0
   		\end{matrix}  & 0  \\
\hline
0 & 0 & I_{n-i-1}
\end{array} \right),\\ \text{ where } bc\neq 0 \hspace*{0.2cm} \text{for} \hspace*{0.2cm} 1\leq i\leq n-1.$
\end{itemize}
\end{theorem}

\vspace*{0.1cm}

\section{Local Representations of $FVB_2$ into $\mathrm{GL}_2(\mathbb{C})$}

\vspace*{0.1cm}

In this section, we determine all local representations of the flat virtual braid group, $FVB_2$, into $\mathrm{GL}_2(\mathbb{C})$. Then, we study the irreducibility and the faithfulness of the determined representations.

\vspace*{0.1cm}

\begin{theorem} \label{n2}
Let $\lambda: FVB_2 \rightarrow \mathrm{GL}_2(\mathbb{C})$ be a non-trivial local representation of $FVB_2$ into $\mathrm{GL}_2(\mathbb{C})$. Then, $\lambda$ has one of the following twelve types:
\begin{itemize}
\item[(1)]  $\lambda_1(\sigma_1) =\left( \begin{array}{c@{}}
   \begin{matrix}
   		-d & b\\
   		\frac{1-d^2}{b} & d
   		\end{matrix}
\end{array} \right)\hspace*{0.15cm} \text{and} \hspace*{0.15cm} \lambda_1(\rho_1) =\left( \begin{array}{c@{}}
   \begin{matrix}
   		-t & y\\
   		\frac{1-t^2}{y} & t
   		\end{matrix}
\end{array} \right),$ where $ b\ne 0, y\neq 0$;
\item[(2)] $\lambda_2(\sigma_1) =\left( \begin{array}{c@{}}
   \begin{matrix}
   		1 & 0\\
   		c & -1
   		\end{matrix}
\end{array} \right)\hspace*{0.15cm} \text{and} \hspace*{0.15cm} \lambda_2(\rho_1) =\left( \begin{array}{c@{}}
   \begin{matrix}
   		-t & y\\
   		\frac{1-t^2}{y} & t
   		\end{matrix}
\end{array} \right),$ where $ y\neq 0$; 
\item[(3)]  $\lambda_3(\sigma_1) =\left( \begin{array}{c@{}}
   \begin{matrix}
   		1 & 0\\
   		0 & 1
   		\end{matrix}
\end{array} \right)\hspace*{0.15cm} \text{and} \hspace*{0.15cm} \lambda_3(\rho_1) =\left( \begin{array}{c@{}}
   \begin{matrix}
   		-t & y\\
   		\frac{1-t^2}{y} & t
   		\end{matrix}
\end{array} \right),$ where $y\neq 0$; 
\item[(4)] $\lambda_4(\sigma_1) =\left( \begin{array}{c@{}}
   \begin{matrix}
   		-d & b\\
   		\frac{1-d^2}{b} & d
   		\end{matrix}
\end{array} \right)\hspace*{0.15cm} \text{and} \hspace*{0.15cm} \lambda_4(\rho_1) =\left( \begin{array}{c@{}}
   \begin{matrix}
   		1 & 0\\
   		z & -1
   		\end{matrix}
\end{array} \right),$ where $b\neq 0$; 
\item[(5)] $\lambda_5(\sigma_1) =\left( \begin{array}{c@{}}
   \begin{matrix}
   		-d & b\\
   		\frac{1-d^2}{b} & d
   		\end{matrix}
\end{array} \right)\hspace*{0.15cm} \text{and} \hspace*{0.15cm} \lambda_5(\rho_1) =\left( \begin{array}{c@{}}
   \begin{matrix}
   		1 & 0\\
   		0 & 1
   		\end{matrix}
\end{array} \right),$ where $ b\neq 0$; 
\item[(6)] $\lambda_6(\sigma_1) =\left( \begin{array}{c@{}}
   \begin{matrix}
   		1 & 0\\
   		c & -1
   		\end{matrix}
\end{array} \right)\hspace*{0.15cm} \text{and} \hspace*{0.15cm} \lambda_6(\rho_1) =\left( \begin{array}{c@{}}
   \begin{matrix}
   		1 & 0\\
   		z & -1
   		\end{matrix}
\end{array} \right)$; 
\item[(7)] $\lambda_7(\sigma_1) =\left( \begin{array}{c@{}}
   \begin{matrix}
   		1 & 0\\
   		c & -1
   		\end{matrix}
\end{array} \right)\hspace*{0.15cm} \text{and} \hspace*{0.15cm} \lambda_7(\rho_1) =\left( \begin{array}{c@{}}
   \begin{matrix}
   		-1 & 0\\
   		z & 1
   		\end{matrix}
\end{array} \right)$; \\
\item[(8)] $\lambda_8(\sigma_1) =\left( \begin{array}{c@{}}
   \begin{matrix}
   		1 & 0\\
   		0 & 1
   		\end{matrix}
\end{array} \right)\hspace*{0.15cm} \text{and} \hspace*{0.15cm} \lambda_8(\rho_1) =\left( \begin{array}{c@{}}
   \begin{matrix}
   		1 & 0\\
   		z & -1
   		\end{matrix}
\end{array} \right)$;  \\
\item[(9)] $\lambda_9(\sigma_1) =\left( \begin{array}{c@{}}
   \begin{matrix}
   		1 & 0\\
   		c & -1
   		\end{matrix}
\end{array} \right)\hspace*{0.15cm} \text{and} \hspace*{0.15cm} \lambda_9(\rho_1) =\left( \begin{array}{c@{}}
   \begin{matrix}
   		1 & 0\\
   		0 & 1
   		\end{matrix}
\end{array} \right)$;  \\
\item[(10)] $\lambda_{10}(\sigma_1) =\left( \begin{array}{c@{}}
   \begin{matrix}
   		1 & 0\\
   		0 & 1
   		\end{matrix}
\end{array} \right)\hspace*{0.15cm} \text{and} \hspace*{0.15cm} \lambda_{10}(\rho_1) =\left( \begin{array}{c@{}}
   \begin{matrix}
   		-1 & 0\\
   		0 & -1
   		\end{matrix}
\end{array} \right)$;
\item[(11)] $\lambda_{11}(\sigma_1) =\left( \begin{array}{c@{}}
   \begin{matrix}
   		1 & 0\\
   		0 & 1
   		\end{matrix}
\end{array} \right)\hspace*{0.15cm} \text{and} \hspace*{0.15cm} \lambda_{11}(\rho_1) =\left( \begin{array}{c@{}}
   \begin{matrix}
   		-1 & 0\\
   		z & 1
   		\end{matrix}
\end{array} \right)$;  \\
\item[(12)] $\lambda_{12}(\sigma_1) =\left( \begin{array}{c@{}}
   \begin{matrix}
   		-1 & 0\\
   		c & 1
   		\end{matrix}
\end{array} \right)\hspace*{0.15cm} \text{and} \hspace*{0.15cm} \lambda_{12}(\rho_1) =\left( \begin{array}{c@{}}
   \begin{matrix}
   		1 & 0\\
   		0 & 1
   		\end{matrix}
\end{array} \right)$. 
\end{itemize}
\end{theorem}

\begin{proof}
Let $\lambda: FVB_2 \rightarrow \mathrm{GL}_2(\mathbb{C})$ be a non-trivial local representation of the flat virtual braid group, $FVB_2$, into $\mathrm{GL}_2(\mathbb{C})$.

\vspace*{0.1cm}

\noindent Set $\lambda(\sigma_1) =\left( \begin{array}{c@{}}
   \begin{matrix}
   		a & b\\
   		c & d
   		\end{matrix}
\end{array} \right)\hspace*{0.2cm} \text{and} \hspace*{0.2cm} \lambda(\rho_1) =\left( \begin{array}{c@{}}
   \begin{matrix}
   		x & y\\
   		z & t
   		\end{matrix}
\end{array} \right),  a, b, c, d, x, y, z, t \in \mathbb{C}$ 
$ad-bc \neq 0, xt-yz\neq 0.$ The only relations between the generators of $FVB_2$ are $\sigma_1^2=1$ and $\rho_1^2=1$. This implies that $(\lambda(\sigma_1))^2=I_2$ and $(\lambda(\rho_1))^2=I_2$. Applying these relations, we get the following system of eight equations and eight unknowns.
\begin{equation} \label{eqq5}
a^2+bc-1=0,
\end{equation} 
\begin{equation} \label{eqq6}
db+ab=0,
\end{equation} 
\begin{equation} \label{eqq7}
dc+ac=0,
\end{equation}
\begin{equation} \label{eqq8}
d^2+bc-1= 0,
\end{equation} 
\begin{equation} \label{eqq1}
x^2+yz-1=0,
\end{equation} 
\begin{equation} \label{eqq2}
ty+xy=0,
\end{equation} 
\begin{equation} \label{eqq3}
tz+xz=0,
\end{equation}
\begin{equation} \label{eqq4}
t^2+yz-1= 0.
\end{equation}
To solve this system of equations, we consider at the beginning Equations $(\ref{eqq5}),\ldots, (\ref{eqq8})$. We get from Equation (\ref{eqq5}) and Equation (\ref{eqq8}) that $a^2=d^2$, which gives that $a=\pm d$. We consider four cases in the following.
\begin{itemize}
\item $a=-d, b=0$. From (\ref{eqq5}) and (\ref{eqq8}), we get $a^2=1$ and $d^2=1$, and so $a=\pm1$ and $d=\mp 1$, and we have in this case
$$\lambda(\sigma_1) =\pm \left( \begin{array}{c@{}}
   \begin{matrix}
   		1 & 0\\
   		c & -1
   		\end{matrix}
\end{array} \right).$$
\item $a=-d, b\neq 0$. From Equation (\ref{eqq8}), we get $c=\frac{1-d^2}{b}$, and so we have
$$\lambda(\sigma_1) =\pm \left( \begin{array}{c@{}}
   \begin{matrix}
   		-d & b\\
   		\frac{1-d^2}{b} & d
   		\end{matrix}
\end{array} \right).$$
\item $a=d=0$. From Equation (\ref{eqq5}), we get $bc=1$, and so, we have
$$\lambda(\sigma_1) =\pm \left( \begin{array}{c@{}}
   \begin{matrix}
   		0 & b\\
   		\frac{1}{b} & 0
   		\end{matrix}
\end{array} \right).$$
Notice that this form is a special case for the previous form when $d=0$.
\item $a=d\neq 0$. From Equation (\ref{eqq6}) and Equation (\ref{eqq7}), we get $b=c=0$, and so, from Equation (\ref{eqq5}) and Equation (\ref{eqq8}), we get $a=d=\pm1$. Thus, we have
$$\lambda(\sigma_1) =\pm \left( \begin{array}{c@{}}
   \begin{matrix}
   		1 & 0\\
   		0 & 1
   		\end{matrix}
\end{array} \right).$$
\end{itemize}
Hence, $\lambda(\sigma_1)$ have 3 forms, which are:
$$\lambda(\sigma_1) =\pm \left( \begin{array}{c@{}}
   \begin{matrix}
   		1 & 0\\
   		c & -1
   		\end{matrix}
\end{array} \right), \lambda(\sigma_1) =\pm \left( \begin{array}{c@{}}
   \begin{matrix}
   		-d & b\\
   		\frac{1-d^2}{b} & d
   		\end{matrix}
\end{array} \right) \text{ and } \lambda(\sigma_1) =\pm \left( \begin{array}{c@{}}
   \begin{matrix}
   		1 & 0\\
   		0 & 1
   		\end{matrix}
\end{array} \right).$$
Again, since the form of Equations $(\ref{eqq1}),\ldots, (\ref{eqq4})$ is the same of Equations $(\ref{eqq5}),\ldots, (\ref{eqq8})$, we get that $\lambda(\rho_1)$ have 3 forms also, which are:
$$\lambda(\rho_1) =\pm \left( \begin{array}{c@{}}
   \begin{matrix}
   		1 & 0\\
   		z & -1
   		\end{matrix}
\end{array} \right), \lambda(\rho_1) =\pm \left( \begin{array}{c@{}}
   \begin{matrix}
   		-t & y\\
   		\frac{1-t^2}{y} & t
   		\end{matrix}
\end{array} \right) \text{ and } \lambda(\rho_1) =\pm \left( \begin{array}{c@{}}
   \begin{matrix}
   		1 & 0\\
   		0 & 1
   		\end{matrix}
\end{array} \right).$$
Joining each case of the forms of $\lambda(\sigma_1)$ together with each case of the forms of $\lambda(\rho_1)$, and eliminating the equivalent terms, we get the required results. 
\end{proof}

Now, we study the irreducibility of any representation $\lambda$ of type $\lambda_i$ for $1\leq i \leq 12$. 

\begin{theorem}
The following statements hold true.
\begin{itemize}
\item[(1)] Every representation of type $\lambda_1$ is irreducible if and only if $b(t\pm1)\neq y(d\pm1)$;
\item[(2)] Every representationof type $\lambda_2$ is irreducible if and only if $yc\neq 2(t\pm1)$;
\item[(3)] Every representation of type $\lambda_3$ is irreducible if and only if $y\neq t\pm1$;
\item[(4)] Every representation of type $\lambda_4$ is irreducible if and only if $bz\neq 2(d\pm1)$;
\item[(5)] Every representation of type $\lambda_5$ is irreducible if and only if $b\neq d\pm1$;
\item[(6)] Every representation of type $\lambda_i$, $6\leq i\leq 12$, is reducible.
\end{itemize}
\end{theorem}

\begin{proof}
Let $\lambda: FVB_2 \rightarrow \mathrm{GL}_2(\mathbb{C})$ be a local representation of the flat virtual braid group, $FVB_2$, of type $\lambda_i$, $1\leq i \leq 12$. The representation $\lambda$ is reducible if and only if the matrices $\lambda(\sigma_1)$ and $\lambda(\rho_1)$ have a common eigenvector.

\vspace*{0.1cm}

\noindent We consider the following cases:

\vspace*{0.1cm}

\noindent \textbf{Case 1.} $\lambda$ is of type $\lambda_1$.In this case, the eigenvectors of the matrix $\lambda(\sigma_1)$ are $(\frac{b}{-1+d},1)$ and $(\frac{b}{1+d},1)$, and those of the matrix $\lambda(\rho_1)$ are $(\frac{y}{-1+t},1)$ and $(\frac{y}{1+t},1)$.   

\vspace*{0.1cm}

\noindent Direct computations show that the representation is irreducible if and only if
\begin{center}
$b(t\pm1)\neq y(d\pm1)$. 
\end{center}

\noindent \textbf{Case 2.} $\lambda$ is of type $\lambda_2$. In this case, the eigenvectors of the matrix $\lambda(\sigma_1)$ are $(0,1)$ and $(\frac{2}{c},1)$, and those of the matrix $\lambda(\rho_1)$ are
$(\frac{y}{-1+t},1)$ and $(\frac{y}{1+t},1)$.   

\vspace*{0.1cm}

\noindent Direct computation show that the representation is irreducible if and only if
\begin{center}
$yc\neq 2(t\pm1)$.  
\end{center}

\noindent \textbf{Case 3.} $\lambda$ is of type $\lambda_3$. In this case, the eigenvectors of the matrix $\lambda(\sigma_1)$ are $(1,0)$ and $(0,1)$, and those of the matrix $\lambda_1(\rho_1)$ are $(\frac{y}{-1+t},1)$ and $(\frac{y}{1+t},1)$.   

\vspace*{0.1cm}

\noindent Direct computation show that the representation is irreducible if and only if
\begin{center}
$y\neq t\pm1$.
\end{center}

\noindent \textbf{Case 4.} $\lambda$ is of type $\lambda_4$. In this case, the eigenvectors of the matrix $\lambda_1(\sigma_1)$ are $(\frac{b}{d-1},1)$ and $(\frac{b}{d+1},1)$, and those of the matrix $\lambda_1(\rho_1)$ are $(0,1)$ and $(\frac{2}{z},1)$.

\vspace*{0.1cm}

\noindent Direct computation show that the representation is irreducible if and only if
\begin{center}
$bz\neq 2(d\pm1)$.
\end{center}

\noindent \textbf{Case 5.} $\lambda$ is of type $\lambda_5$. In this case, the eigenvectors of the matrix $\lambda(\sigma_1)$ are $(\frac{b}{d-1},1)$ and $(\frac{b}{d+1},1)$, and those of the matrix $\lambda_1(\rho_1)$ are $(0,1)$ and $(1,0)$.   

\vspace*{0.1cm}

\noindent Direct computation show that the representation is irreducible if and only if
\begin{center}
$b\neq d\pm1$.   
\end{center}

\noindent \textbf{Case 6.} $\lambda$ is of type $\lambda_i$, $6\leq i\leq 12$. In this case, it is clear that the subspace $S=<(0,1)>$ is invariant under the images of the generators $\sigma_1$ and $\rho_1$ of the representation $\lambda$. This implies that the representation $\lambda$ is reducible.
\end{proof}

Now, we study the faithfulness of any representation $\lambda$ of type $\lambda_i$ for $1\leq i \leq 12$. 

\begin{theorem}
The following statements hold true.
\begin{itemize}
\item[(1)] Every representation of type $\lambda_i$, $i=3,5,8,9,10,11,12,$ is unfaithful.
\item[(2)] Every representation of type $\lambda_1$ is unfaithful if ($b=y$ and $d=t$) or ($b=-y$ and $d=-t$) .
\item[(3)] Every representation of type $\lambda_2$ is unfaithful if $t=\frac{cy}{2}$.
\item[(4)] Every representation of type $\lambda_4$ is unfaithful if $d=\frac{bz}{2}$.
\item[(5)] Every representation of type $\lambda_6$ is faithful if and only if $z\neq c$.
\item[(6)] Every representation of type $\lambda_7$ is faithful if and only if $z\neq -c$.
\end{itemize}
\end{theorem}

\begin{proof} 
Let $\lambda: FVB_2 \rightarrow \mathrm{GL}_2(\mathbb{C})$ be a local representation of the flat virtual braid group, $FVB_2$, of type $\lambda_i$, $1\leq i \leq 12$. We consider the following cases:
\begin{itemize}
\item[\textbf{Case 1.}] $\lambda$ is a representation of type $\lambda_3$, $\lambda_5$ $\lambda_8$, $\lambda_9$, $\lambda_{10}$, $\lambda_{11}$ or $\lambda_{12}$. We notice in this case that $\lambda(\sigma_1)=I_2$ if $\lambda$ is of type $\lambda_3$, $\lambda_8$, $\lambda_{10}$ or $\lambda_{12}$ and $\lambda(\rho_1)=I_2$ if $\lambda$ is of type $\lambda_5$, $\lambda_9$ or $\lambda_{12}$. This implies that any representation $\lambda$ of type $\lambda_3$, $\lambda_5$ $\lambda_8$, $\lambda_9$, $\lambda_{10}$, $\lambda_{11}$ or $\lambda_{12}$ is unfaithful.

\vspace*{0.1cm}

\item[\textbf{Case 2.}] $\lambda$ is a representation of type $\lambda_1$. First, let $b=y$ and $d=t$. Direct computations show that $\lambda(\rho_1\sigma_1)=I_2$ where $\rho_1\sigma_1$ is a non-trivial element in $FVB_2$, and this implies that $\lambda$ is unfaithful. Second, let $b=-y$ and $d=-t$. Direct computations show that $\lambda((\rho_1\sigma_1)^2)=I_2$ with $(\rho_1\sigma_1)^2$ is a non-trivial element in $FVB_2$, and this implies that $\lambda$ is unfaithful.

\vspace*{0.1cm}

\item[\textbf{Case 3.}] $\lambda$ is a representation of type $\lambda_2$. Let $t=\frac{cy}{2}$. Direct computations show that $\lambda((\rho_1\sigma_1)^4)=I_2$ with $(\rho_1\sigma_1)^4$ is a non-trivial element in $FVB_2$, and this implies implies that $\lambda$ is unfaithful.

\vspace*{0.1cm}

\item[\textbf{Case 4.}] $\lambda$ is a representation of type $\lambda_4$. Let $t=\frac{bz}{2}$. Direct computations show that $\lambda((\sigma_1\rho_1)^4)=I_2$ with $(\sigma_1\rho_1)^4$ is a non-trivial element in $FVB_2$, and this implies that $\lambda$ is unfaithful.

\vspace*{0.1cm}

\item[\textbf{Cases 5, 6.}] $\lambda$ is a representation of type $\lambda_6$ or $\lambda_7$. We deal, without loss of generality, with representations of type $\lambda_6$, and a similar proof could be done for representations of type $\lambda_7$. Now, for the necessary condition, suppose that $c=z$. We see that $\lambda(\rho_1\sigma_1)=I_2$ with $\rho_1\sigma_1$ is a non-trivial element in $FVB_2$, and so $\lambda$ is unfaithful. For the sufficient condition, suppose that $\lambda$ is unfaithful, then there exists a non-trivial element $w\in FVB_2$ such that $\lambda_6(w)=I_2$ . As $\sigma_1^2=\rho_1^2=1$, it follows that $w$ could be one of the following four types of elements in $FVB_2$.
\begin{itemize}
\item[•] $w_1=(\sigma_1\rho_1)^n$ for some $n \in \mathbb{N}^*$
\item[•] $w_2=(\rho_1\sigma_1)^n$ for some $n \in \mathbb{N}^*$
\item[•] $w_3=\sigma_1(\rho_1\sigma_1)^n$ for some $n \in \mathbb{N}^*$
\item[•] $w_4=(\rho_1\sigma_1)^n\rho_1$ for some $n \in \mathbb{N}^*$.
\end{itemize}
The last two elements, $w_3$ and $w_4$, couldn't be in $\ker(\lambda)$, since direct computations show that
$$\lambda(w_3) =\left( \begin{array}{c@{}}
   \begin{matrix}
   		1 & 0\\
   		* & -1
   		\end{matrix}
\end{array} \right) \text{ and } \lambda(w_4) =\left( \begin{array}{c@{}}
   \begin{matrix}
   		1 & 0\\
   		\bullet & -1
   		\end{matrix}
\end{array} \right).$$
For the first two elements, $w_1$ and $w_2$, we suppose, without loss of generality, that $w=w_1=(\sigma_1\rho_1)^n$. Using induction on $n$, we see that $$\lambda(w) =\left( \begin{array}{c@{}}
   \begin{matrix}
   		1 & 0\\
   		n(c-z) & 1
   		\end{matrix}
\end{array} \right).$$ But $w\in \ker(\lambda)$, which gives that $c=z$, as required.
\end{itemize}
\end{proof}

\section{Homogeneous Local Representations of $FVB_n$ into $\mathrm{GL}_n(\mathbb{C})$}

\vspace*{0.1cm}

In this section, we determine all homogeneous local representations of the flat virtual braid group, $FVB_n$, into $\mathrm{GL}_n(\mathbb{C})$, for $n\geq 3$. Then, we study the irreducibility and the faithfulness of the determined representations.

\begin{theorem} \label{Theo}
Consider $n\geq 3$ and let $\gamma: FVB_n \rightarrow \mathrm{GL}_n(\mathbb{C})$ be a non-trivial homogeneous local representation of $FVB_n$. Then, $\gamma$ has one of the following two types: 
\begin{itemize}
\item[(1)] $\gamma_1: FVB_n\to \mathrm{GL}_{n}(\mathbb{C})$ such that
$$\gamma_1(\sigma_i)=I_n \text{ and } \gamma_1(\rho_i) =\left( \begin{array}{c|@{}c|c@{}}
   \begin{matrix}
     I_{i-1} 
   \end{matrix} 
      & 0 & 0 \\
      \hline
    0 &\hspace{0.2cm} \begin{matrix}
   		0 & y\\
   		\frac{1}{y} & 0
   		\end{matrix}  & 0  \\
\hline
0 & 0 & I_{n-i-1}
\end{array} \right),$$ where $y \in \mathbb{C}^*.$ \\
\item[(2)] $\gamma_2: FVB_n\to \mathrm{GL}_{n}(\mathbb{C})$ such that
$$\gamma_2(\sigma_i)=\left( \begin{array}{c|@{}c|c@{}}
   \begin{matrix}
     I_{i-1} 
   \end{matrix} 
      & 0 & 0 \\
      \hline
    0 &\hspace{0.2cm} \begin{matrix}
   		0 & b\\
   		\frac{1}{b} & 0
   		\end{matrix}  & 0  \\
\hline
0 & 0 & I_{n-i-1}
\end{array} \right) \text{ and } \gamma_2(\rho_i) =\left( \begin{array}{c|@{}c|c@{}}
   \begin{matrix}
     I_{i-1} 
   \end{matrix} 
      & 0 & 0 \\
      \hline
    0 &\hspace{0.2cm} \begin{matrix}
   		0 & y\\
   		\frac{1}{y} & 0
   		\end{matrix}  & 0  \\
\hline
0 & 0 & I_{n-i-1}
\end{array} \right),$$ where $b,y \in \mathbb{C}^*.$
\end{itemize}
\end{theorem}

\begin{proof}
Set $$\gamma(\sigma_i)=\left( \begin{array}{c|@{}c|c@{}}
   \begin{matrix}
     I_{i-1} 
   \end{matrix} 
      & 0 & 0 \\
      \hline
    0 &\hspace{0.2cm} \begin{matrix}
   		a & b\\
   		c & d
   		\end{matrix}  & 0  \\
\hline
0 & 0 & I_{n-i-1}
\end{array} \right)\ \text{and} \ \gamma(\rho_i) =\left( \begin{array}{c|@{}c|c@{}}
   \begin{matrix}
     I_{i-1} 
   \end{matrix} 
      & 0 & 0 \\
      \hline
    0 &\hspace{0.2cm} \begin{matrix}
   		x & y\\
   		z & t
   		\end{matrix}  & 0  \\
\hline
0 & 0 & I_{n-i-1}
\end{array} \right),$$
$\text{where } a,b,c,d,x,y,z,t\in \mathbb{C}, ad-bc\neq 0, xt-yz\neq 0,\  \text{for all} \ 1\leq i \leq n-1.$ Considering the relations (\ref{eqs1}), (\ref{eqs2}), \ldots, (\ref{eqs8}) of $FVB_n$, we obtain the following system of twenty four equations and eight unknown.
\begin{equation}\label{e17}
-1+a^2+bc=0,
\end{equation}
\begin{equation}\label{e18}
b(a+d)=0,
\end{equation}
\begin{equation}\label{e19}
c(a+d)=0,
\end{equation}
\begin{equation}\label{e20}
-1+bc+d^2=0,
\end{equation}
\begin{equation}\label{e21}
 a(-1+a+bc)=0,
\end{equation}
\begin{equation}\label{e22}
abd=0,
\end{equation}
\begin{equation}\label{e23}
acd=0,
\end{equation}
\begin{equation}\label{e24}
ad(a-d)=0,
\end{equation}
\begin{equation}\label{e25}
d(1-bc-d)=0,
\end{equation}
\begin{equation}\label{e26}
x(-1+x+yz)=0,
\end{equation}
\begin{equation}\label{e27}
txy=0,
\end{equation}
\begin{equation}\label{e28}
txz=0,
\end{equation}
\begin{equation}\label{e29}
xt(x-t)=0,
\end{equation}
\begin{equation}\label{e30}
t(1-t-yz)=0,
\end{equation}
\begin{equation}\label{e31}
x(-1+a+cy)=0,
\end{equation}
\begin{equation}\label{e32}
x(b-y+dy)=0,
\end{equation}
\begin{equation}\label{e33}
ctx=0,
\end{equation}
\begin{equation}\label{e34}
xt(a-d)=0,
\end{equation}
\begin{equation}\label{e35}
t(-b+y-ay)=0,
\end{equation}
\begin{equation}\label{e36}
t(1-d-cy)=0,
\end{equation}
\begin{equation}\label{e37}
-1+x^2+yz=0,
\end{equation}
\begin{equation}\label{e38}
y(t+x)=0,
\end{equation}
\begin{equation}\label{e39}
z(t+x)=0,
\end{equation}
\begin{equation}\label{e40}
-1+t^2+yz=0.
\end{equation}
First of all, we show that   
$$\gamma(\rho_i) =\left( \begin{array}{c|@{}c|c@{}}
   \begin{matrix}
     I_{i-1} 
   \end{matrix} 
      & 0 & 0 \\
      \hline
    0 &\hspace{0.2cm} \begin{matrix}
   		0 & y\\
   		\frac{1}{y} & 0
   		\end{matrix}  & 0  \\
\hline
0 & 0 & I_{n-i-1}
\end{array} \right) \text{ for all } 1\leq i\leq n-1.$$
From Equation (\ref{e29}), we have three possible cases: $x=t\neq 0, x=0,$ or $t=0.$ We consider in the following each case separately. 
\begin{itemize}
\item[(1)] Suppose first that $x=t\neq 0$. In this case, we get by Equation (\ref{e27}) and Equation (\ref{e28}) that $y=z=0$. Now, from Equation (\ref{e26}) and Equation (\ref{e30}), we get that $x=t=1$. From the other side, we see from Equation (\ref{e34}), that $a=d$, and so, we get by Equation (\ref{e18}) and Equation (\ref{e19}) that $b=c=0$. Thus, by Equation (\ref{e21}) and Equation (\ref{e25}), we get $a=d=1$. Therefore, $\gamma$ here is the trivial representation. This case is eliminated by the statement of our theorem.
\item[(2)] Suppose now $x=0$. In this case, we get from Equation (\ref{e37}) that $yz=1$, and so we get from Equation (\ref{e40}) that $t=0$, as required.
\item[(3)] Suppose now $t=0$. In this case, we get from Equation (\ref{e40}) that $yz=1$, and so we get from Equation (\ref{e37}) that $x=0$, as required.
\end{itemize}

Now, we discuss the forms of $\gamma(\sigma_i)$. From Equation (\ref{e24}), we have three possible cases: $a=d\neq 0, a=0,$ or $d=0.$ We consider in the following each case separately. 
\begin{itemize}
\item[(1)] If $a=d\neq 0$, then, by Equation (\ref{e22}) and Equation (\ref{e23}), we get $b=c=0$. Then, by Equation (\ref{e21}) and Equation (\ref{e24}), we get $a=d=1$. This case is the representation $\gamma_1$.
\item[(2)] Now, if $a=0$, then, by Equation (\ref{e17}), we get $bc=1$ and, by Equation (\ref{e18}), we get $d=0$. This case is $\gamma_2$.
\item[(3)] Similarly, if $d=0$, then, by Equation (\ref{e20}), we get $bc=1$ and, by Equation (\ref{e18}), we get $a=0$. This case is $\gamma_2$.  
\end{itemize}
Therefore, $\gamma$ is equivalent to $\gamma_1$ and $\gamma_2$, as required.
\end{proof}

Now, we study the irreducibility of any non-trivial local representation of $FVB_n$ into $\mathrm{GL}_n(\mathbb{C}) $ for all $n\geq 3$.

\begin{theorem}
Let $\gamma:FVB_n \rightarrow \mathrm{GL}_n(\mathbb{C})$, $n\geq 3$, be a homogeneous local representation of the flat virtual braid group $FVB_n$ of type $\gamma_1$. Then $\gamma$ is reducible for $n\geq 6$.  
\end{theorem}

\begin{proof}
Consider the representation $\Tilde{\gamma}:B_n\to \mathrm{GL}_{n}(\mathbb{C})$, of the braid group, $B_n$, defined  by $\Tilde{\gamma}(\sigma_j)=\gamma(\rho_j)$ for all $1\leq j \leq n-1$. The representation $\Tilde{\gamma}$ is a homogeneous local representation of type $\varphi_2$ defined, in \cite{17}, by Mikhalchishina for $d=0$ and $c=\frac{1}{y}$. (See \cite{17}, Corollary to Theorem 1).

\vspace*{0.1cm}

The homogeneous local representations of type $\varphi_2$ are reducible for all $n\geq 6$ (\cite{20}, Theorem 3.1). This implies that there exists a proper subspace $S$ of $\mathbb{C}^n $ which is invariant under the images of the generators $\sigma_j$ of the representation $\tilde{\gamma}$. Thus, $S$ is invariant under the images of the generators $\rho_j$ of the representation $\gamma$. Now, Since $\gamma(\sigma_j)=I_n$ for all $1\leq j \leq n-1$, it follows that $S$ is also invariant under the images of the generators $\sigma_j$ of the representation $\gamma$.

\vspace*{0.1cm}

Therefore, the representation $\gamma$ defined over the flat virtual braid group $FVB_n$ is reducible for $n\geq 6$.
\end{proof}
 
\begin{theorem}
Let $\gamma:FVB_n \rightarrow \mathrm{GL}_n(\mathbb{C})$, $n\geq 3$, be a representation of the flat virtual braid group, $FVB_n$, of type $\gamma_2$, then $\gamma$ is irreducible if and only if $b\neq y$.
\end{theorem}

\begin{proof}
We introduce a representation equivalent to $\gamma_2$, denoted by $\alpha_2$, defined as follows.  Let
$P=Diag\big(b^{n-1},\, b^{n-2},\,\ldots,\, b,\,1\big)$, where $Diag(a_1,a_2,\ldots,a_n)$ is a diagonal $n\times n$ matrix with $a_i=a_{ii}$. Define
\[
\alpha_2(g)=P^{-1}\gamma_2(g)P
\ \text{ for all } g\in FVB_n.
\]
A direct computation shows that for the generators $\sigma_i$ and $\rho_i$, where $1\le i\le n-1$, we obtain
\[
\alpha_2(\sigma_i)=\left( \begin{array}{c|@{}c|c@{}}
   \begin{matrix}
     I_{i-1} 
   \end{matrix} 
      & 0 & 0 \\
      \hline
    0 &\hspace{0.2cm} \begin{matrix}
   		0 & 1\\
   		1 & 0
   		\end{matrix}  & 0  \\
\hline
0 & 0 & I_{n-i-1}
\end{array} \right)
\]
and
\[
\alpha_2(\rho_i)=
\left( \begin{array}{c|@{}c|c@{}}
   \begin{matrix}
     I_{i-1} 
   \end{matrix} 
      & 0 & 0 \\
      \hline
    0 &\hspace{0.2cm} \begin{matrix}
   		0 & \frac{y}{b}\\
   		\frac{b}{y} & 0
   		\end{matrix}  & 0  \\
\hline
0 & 0 & I_{n-i-1}
\end{array} \right)
\]
where $b,y\in\mathbb{C}^*$.

For the necessary condition, assume first that $b=y$. Then the vector $(1,1,\ldots,1)^T$
is fixed by both $\alpha_2(\sigma_i)$ and $\alpha_2(\rho_i)$ for every $1\le i\le n-1$, showing that $\alpha_2$ is reducible. Since $\alpha_2$ is equivalent to $\gamma_2$, it follows that $\gamma_2$ is also reducible.

For the sufficient condition, suppose that $b\neq y$, and assume toward a contradiction that $\gamma_2$ is reducible. Since equivalent representations have the same reducibility behavior, $\alpha_2$ must also be reducible. Hence there exists a nonzero proper subspace $U\subset \mathbb{C}^n$ that is invariant under $\alpha_2$. Choose a nonzero vector $u=(u_1,u_2,\ldots,u_n)^T\in U.$ For each $1\le i\le n-1$, we have $\alpha_2(\sigma_i)u-u
=(u_{i+1}-u_i)(e_i-e_{i+1})\in U,$
where $e_1,\ldots,e_n$ denote the standard basis vectors of $\mathbb{C}^n$. We may select $u$ such that $u_j\neq u_{j+1}$ for some index $j$. Indeed, if $u_i=u_{i+1}$ for all $i$, then $u$ would be a scalar multiple of $(1,1,\ldots,1)^T,$ which cannot generate an invariant subspace because $U$ is stable under $\alpha_2(\rho_i)$ and $b\neq y$. Therefore, $e_j-e_{j+1}\in U$
for some $j$. Applying the operators $\alpha_2(\sigma_{j+1})$ and $\alpha_2(\sigma_{j-1})$ successively, we deduce that 
$e_{j+1}-e_{j+2}\in U
\ \text{ and }\
e_{j-1}-e_j\in U.$ Repeating this argument yields
\begin{equation}\label{eq:diffs}
e_i-e_{i+1}\in U
\ \text{ for all } 1\le i\le n-1.
\end{equation}
Next, observe that no basis vector $e_i$ can belong to $U$. Indeed, if one $e_i$ were contained in $U$, then relation \eqref{eq:diffs} would imply that all basis vectors belong to $U$, forcing 
$U=\mathbb{C}^n,$ contrary to the assumption that $U$ is proper. Now consider the vector
$\alpha_2(\rho_1)(e_1-e_2)
+\frac{y}{b}(e_1-e_2).$ A direct computation gives $\left(\frac{b}{y}-\frac{y}{b}\right)e_2\in U.$
Since $e_2\notin U$, we must have $\frac{b}{y}-\frac{y}{b}=0.$ Hence $ b=\pm y.$ Because $b\neq y$, we necessarily obtain $b=-y.$
In this case, we can easily see that $\alpha_2(\sigma_1)(e_2-e_3)
-\alpha_2(\rho_1)(e_2-e_3)
=2e_1\in U.$
Therefore $e_1\in U$, contradicting the fact that no standard basis vector belongs to $U$. This contradiction proves that $\alpha_2$ is irreducible whenever $b\neq y$. Since $\alpha_2$ is equivalent to $\gamma_2$, we conclude that $\gamma_2$ is irreducible as well.
\end{proof}

Now, we study the faithfulness of every nontrivial homogeneous local representation of $FVB_n$ for $n\geq 3$.

\begin{theorem} \label{Thmmm}
The following statements hold true.
\begin{itemize}
\item[(1)] Every representation of type $\gamma_1$ is unfaithful.
\item[(2)] Every representation of type $\gamma_2$ is unfaithful if $y=b$.
\end{itemize}
\end{theorem}

\begin{proof}
We consider each case separately.
\begin{itemize}
\item[(1)] Let $\gamma$ be a representation of type $\gamma_1$. 
The restriction of $\gamma$ to $B_n$ is the identity map $id: B_n \rightarrow \mathrm{GL}_{n+1}(\mathbb{C})$, which is known to be an unfaithful representation. Thus, $\gamma$ is unfaithful.
\item[(2)] Let $\gamma$ be a representation of type $\gamma_2$. For $y=b$, $\gamma_2(\rho_i\sigma_i)=I_n$, for all $1\leq i \leq n-1$, with $\rho_i\sigma_i$ is a non-trivial element. This implies that $\gamma_2$ is unfaithful.
\end{itemize}
\end{proof}

\vspace*{-0.3cm}

\section{Homogeneous Local Representations of $FVB_n$ into $\mathrm{GL}_{n+1}(\mathbb{C})$}

\vspace*{0.1cm}

In this section, we determine all homogeneous local representations of the flat virtual braid group, $FVB_n$, into $\mathrm{GL}_{n+1}(\mathbb{C})$, for all $n\geq 4$. Then we study the irreducibility and the faithfulness of the determined representations.

\begin{theorem}\label{repVBn}
Let $\delta :FVB_n\to \mathrm{GL}_{n+1}(\mathbb{C})$ be a non-trivial homogeneous local representation of $FVB_n$, $n\geq 4$. Then, $\delta$ has one of the following eight types: 
\begin{itemize}
\item[(1)] $\delta_1(\sigma_i)=I_{n+1} \text{ and }\ \delta_1(\rho_i)=\left( \begin{array}{c|@{}c|c@{}}
   \begin{matrix}
     I_{i-1} 
   \end{matrix} 
      & 0 & 0 \\
      \hline
    0 &\hspace{0.2cm} \begin{matrix}
     1& x & 0\\
     0 & -1 & 0\\
     0 & \frac{1}{x} & 1
   		\end{matrix}  & 0  \\
\hline
0 & 0 & I_{n-i-1}
\end{array} \right),\text{ where } x\neq 0$,  for all $1\leq i \leq n-1$,\\
\item[(2)] $\delta_2(\sigma_i)=I_{n+1} \text{ and }\ \delta_2(\rho_i)=\left( \begin{array}{c|@{}c|c@{}}
   \begin{matrix}
     I_{i-1} 
   \end{matrix} 
      & 0 & 0 \\
      \hline
    0 &\hspace{0.2cm} \begin{matrix}
      1& 0 & 0\\
     \frac{1}{x} & -1 & x\\
     0 & 0 & 1
   		\end{matrix}  & 0  \\
\hline
0 & 0 & I_{n-i-1}
\end{array} \right),\text{ where } x\neq 0,$ for all $1\leq i \leq n-1$,\\
\item[(3)] $\delta_3(\sigma_i)=I_{n+1} \text{ and }\ \delta_3(\rho_i)=\left( \begin{array}{c|@{}c|c@{}}
   \begin{matrix}
     I_{i-1} 
   \end{matrix} 
      & 0 & 0 \\
      \hline
    0 &\hspace{0.2cm} \begin{matrix}
     0& x & 0\\
     \frac{1}{x} & 0 & 0\\
     0 & 0 & 1
   		\end{matrix}  & 0  \\
\hline
0 & 0 & I_{n-i-1}
\end{array} \right),\text{ where } x\neq 0,$ for all $1\leq i \leq n-1$,\\
\item[(4)] $\delta_4(\sigma_i)=I_{n+1} \text{ and }\ \delta_4(\rho_i)=\left( \begin{array}{c|@{}c|c@{}}
   \begin{matrix}
     I_{i-1} 
   \end{matrix} 
      & 0 & 0 \\
      \hline
    0 &\hspace{0.2cm} \begin{matrix}
     1& 0 & 0\\
     0 & 0 & x\\
     0 & \frac{1}{x} & 0
   		\end{matrix}  & 0  \\
\hline
0 & 0 & I_{n-i-1}
\end{array} \right),\text{ where } x\neq 0,$ for all $1\leq i \leq n-1$,\\
\item[(5)] $\delta_5(\sigma_i)=\left( \begin{array}{c|@{}c|c@{}}
   \begin{matrix}
     I_{i-1} 
   \end{matrix} 
      & 0 & 0 \\
      \hline
    0 &\hspace{0.2cm} \begin{matrix}
     0&  \frac{1}{x} & 0\\
     x & 0 & 0\\
     0 & 0 & 1
   		\end{matrix}  & 0  \\
\hline
0 & 0 & I_{n-i-1}
\end{array} \right) \text{ and }\ \delta_5(\rho_i)=\left( \begin{array}{c|@{}c|c@{}}
   \begin{matrix}
     I_{i-1} 
   \end{matrix} 
      & 0 & 0 \\
      \hline
    0 &\hspace{0.2cm} \begin{matrix}
     0& y & 0\\
     \frac{1}{y} & 0 & 0\\
     0 & 1 & 1
   		\end{matrix}  & 0  \\
\hline
0 & 0 & I_{n-i-1}
\end{array} \right),$ where $x,y\neq 0,$ for all $1\leq i \leq n-1$,\\
\item[(6)]$\delta_6(\sigma_i)=\left( \begin{array}{c|@{}c|c@{}}
   \begin{matrix}
     I_{i-1} 
   \end{matrix} 
      & 0 & 0 \\
      \hline
    0 &\hspace{0.2cm} \begin{matrix}
     1&  0 & 0\\
     0 & 0 & \frac{1}{x}\\
     0 & x & 0
   		\end{matrix}  & 0  \\
\hline
0 & 0 & I_{n-i-1}
\end{array} \right) \text{ and }\ \delta_6(\rho_i)=\left( \begin{array}{c|@{}c|c@{}}
   \begin{matrix}
     I_{i-1} 
   \end{matrix} 
      & 0 & 0 \\
      \hline
    0 &\hspace{0.2cm} \begin{matrix}
     1& 0 & 0\\
     0 & 0 & y\\
     0 & \frac{1}{y} & 0
   		\end{matrix}  & 0  \\
\hline
0 & 0 & I_{n-i-1}
\end{array} \right),$ where $x,y\neq 0,$ for all $1\leq i \leq n-1$,\\
\item[(7)] $\delta_7(\sigma_i)=\left( \begin{array}{c|@{}c|c@{}}
   \begin{matrix}
     I_{i-1} 
   \end{matrix} 
      & 0 & 0 \\
      \hline
    0 &\hspace{0.2cm} \begin{matrix}
     1&  x & 0\\
     0 & -1 & 0\\
     0 & \frac{1}{x} & 1
   		\end{matrix}  & 0  \\
\hline
0 & 0 & I_{n-i-1}
\end{array} \right) \text{ and }\ \delta_7(\rho_i)=\left( \begin{array}{c|@{}c|c@{}}
   \begin{matrix}
     I_{i-1} 
   \end{matrix} 
      & 0 & 0 \\
      \hline
    0 &\hspace{0.2cm} \begin{matrix}
     1& x & 0\\
     0 & -1 & 0\\
     0 & \frac{1}{x} & 0
   		\end{matrix}  & 0  \\
\hline
0 & 0 & I_{n-i-1}
\end{array} \right),$ where $x\neq 0,$ for all $1\leq i \leq n-1$,\\
\item[(8)] $\delta_8(\sigma_i)=\left( \begin{array}{c|@{}c|c@{}}
   \begin{matrix}
     I_{i-1} 
   \end{matrix} 
      & 0 & 0 \\
      \hline
    0 &\hspace{0.2cm} \begin{matrix}
     1&  0 & 0\\
     \frac{1}{x} & -1 & x\\
     0 & 0 & 1
   		\end{matrix}  & 0  \\
\hline
0 & 0 & I_{n-i-1}
\end{array} \right) \text{ and }\ \delta_8(\rho_i)=\left( \begin{array}{c|@{}c|c@{}}
   \begin{matrix}
     I_{i-1} 
   \end{matrix} 
      & 0 & 0 \\
      \hline
    0 &\hspace{0.2cm} \begin{matrix}
     1& 0 & 0\\
     \frac{1}{x} & -1 & x\\
     0 & 0 & 1
   		\end{matrix}  & 0  \\
\hline
0 & 0 & I_{n-i-1}
\end{array} \right),$ where $x\neq 0,$ for all $1\leq i \leq n-1$.
\end{itemize}
\end{theorem}
\begin{proof} 
The proof is similar to the proof of Theorem \ref{Theo}.
\end{proof}

Now, we study the irreducibility of the representations $\delta_i$, $1\leq i \leq 8$.

\begin{theorem}
Let $\delta:FVB_n\to \mathrm{GL}_{n+1}(\mathbb{C})$, be a homogeneous local representation of type $\delta_i$, $1\leq i \leq 8$. Then, $\delta$ is reducible for $n\geq 10$.
\end{theorem}

\begin{proof}
We consider the following two cases:

\vspace*{0.1cm}

\noindent \textbf{Case 1:} In this case, we consider $\delta_i$ for $1\leq i \leq 4$. We define the representation $\Tilde{\delta}:B_n\to \mathrm{GL}_{n+1}(\mathbb{C})$, of the braid group, $B_n$, by $\Tilde{\delta}(\sigma_i)=\delta(\rho_j)$ for all $1\leq j \leq n-1$. We know that, for $n\geq 10$, there are no irreducible complex representations of the braid group, $B_n$, of dimension $n+1$ (See \cite{Sys}, Theorem 6.1). This implies that there exists a proper subspace $S$ of $\mathbb{C}^{n+1}$ which is invariant under the images of the generators $\rho_j$, $1\leq j \leq n-1$, of the representation $\tilde{\delta}$. Consequently, $S$ is invariant under the images of the generators $\rho_j$, $1\leq j \leq n-1$, of the representation $\delta$ of the virtual braid group $FVB_n$. Now, since $\delta(\sigma_j)=I_{n+1}$ for all $1\leq j \leq n-1$, it follows that $S$ is invariant under the images of the generators $\sigma_j$ of the representation $\delta$.\\ 
Therefore, any representation of type $\delta_i$, $1\leq i \leq 4$, defined over $FVB_n$ into $\mathrm{GL}_{n+1}(\mathbb{C})$ is reducible for all $n\geq 10$.

\vspace*{0.1cm}

\noindent \textbf{Case 2:} We now consider $\delta_i$ for $5\leq i \leq 8$. We define the representation $\psi: B_n\to \mathrm{GL}_{n+1}(\mathbb{C})$, of the braid group, $B_n$, by $\psi(\sigma_j)=\delta(\sigma_j)$ for $1\leq j \leq n-1$. If the representation $\delta$ defined over the flat virtual braid group, $FVB_n$, is irreducible, then the representation $\psi$ defined on the braid group, $B_n$, is also irreducible. This contradicts the fact that there are no irreducible complex representations of the braid group, $B_n$, of dimension $n+1$ for $n\geq 10$. (See \cite{Sys}, Theorem 6.1).\\
Therefore, any representation of type $\delta_i$, $5\leq i \leq 6$, defined over $FVB_n$ into $\mathrm{GL}_{n+1}(\mathbb{C})$ is reducible for all $n\geq 10$.
\end{proof}

Now, we study the faithfulness of the representations of type $\delta_i$, $1\leq i\leq 8$. 

\begin{theorem}
The following statements hold true.
\begin{itemize}
\item[(1)] Every representation of type $\delta_i$, $i=1, 2, 3, 4, 7$ or $8$, is unfaithful.
\item[(2)] Every representation of type $\delta_5$ or $\delta_6$ is unfaithful if $x=y$.
\end{itemize}
\end{theorem}

\begin{proof}
Let $\delta: FVB_n \rightarrow \mathrm{GL}_{n+1}(\mathbb{C})$ be a non-trivial homogeneous local representation of $FVB_n$.
We consider the following cases:
\begin{itemize}
\item[(1)] First, we suppose that $\delta$ is a representation of type $\delta_j, 1\leq j \leq 4$. The restriction of $\delta$ to $B_n$ is the identity map $id: B_n \rightarrow \mathrm{GL}_{n+1}(\mathbb{C})$, which is known to be an unfaithful representation. Thus, $\delta$ is unfaithful. Second, we suppose that $\delta$ is a representation of type $\delta_j$, $j=7$ or $8$. We have $\delta_j(\rho_i\sigma_i)=I_{n+1}$ where $\rho_i\sigma_i$ is a non-trivial element of $FVB_n$, $1\leq i \leq n-1$. This implies that $\delta$ is unfaithful.
\item[(2)]  Suppose that $\delta$ is a representation of type $\delta_j$, $j=5$ or $6$. If $x=y$, then $\delta_j(\rho_i\sigma_i)=I_{n+1}$ where $\rho_i\sigma_i$ is a non-trivial element of $FVB_n$, $1\leq i \leq n-1$. This implies that $\delta$ is unfaithful.
\end{itemize}
\end{proof}


\begin{thebibliography}{99}

\bibitem{E.A} E. Artin, \emph{Theorie der Z{\"o}pfe}, {Abhandlungen Hamburg}, 4, 47-72, (1925).

\bibitem{19} V. Bardakov and P. Bellingeri, \emph{On representation of braids as automorphisms of free groups and corresponding linear representations}, {Knot Theory and Its Applications, Contemp. Math.}, Amer. Math. Soc., Providence, 670, 285-298, (2016).

\bibitem{BC} V. Bardakov, B. Chuzhinov, I. E. yanenkov, M. Ivanov, E. Markhinina, T. Nasybullov, S. Panov, N. Singh, S. Vasyutkin, V. Yakhin, and A. Vesnin, \emph{Representations of flat virtual braids which do not preserve the forbidden relations}, {Journal of Knot Theory and Its Ramifications}, 32(14), 2350093, (2023).

\bibitem{BN} V. Bardakov and M. Neshchadim, \emph{A Representation of Virtual Braids by Automorphisms}, {Algebra Logic}, 56, 355–361, (2017).

\bibitem{BNA} V. Bardakov and T. Nasybullov, \emph{Multi-switches and representations of braid groups}, {Journal of Algebra and Its Applications}, 23(3), 2430003, (2024).

\bibitem{5678} S. Bigelow, \emph{The Burau Representation is not faithful for $n=5$}, {Topology} 32, 439-447, (1999).

\bibitem{Big2001} S. Bigelow, \emph{Braid groups are linear}, {J. Amer. Math. Soc.}, 14, 471-486, (2000).
	
\bibitem{1234} J. Birman, \emph{Braids, links and mapping class groups}, {Annals of Mathematical studies}, Princeton University Press, 8, (1974).

\bibitem{1} W. Burau, \emph{Braids, Uber Zopfgruppen and gleichsinnig verdrillte Verkettungen}, {Abh. Math. Semin. Hamburg Univ} 11, 179-186, (1936).

\bibitem{20} M. Chreif, M. Dally, \emph{On the irreducibility of local representations of the Braid group $B_n$}, {Arab. J. Math.}, (2024).

\bibitem{CV} B. Chuzhinov and A. Vesnin, \emph{Representations of Flat Virtual Braids by Automorphisms of Free Group}, {Symmetry} 15(8), (2023).

\bibitem{DRW} C. Delaney, E. Rowell and Z. Wang, \emph{Local Unitary Representations of the Braid Group and Their Applications to Quantum Computing}, {Revista Colombiana de Matem´aticas}, 50, 211-276, (2016).

\bibitem {75} E. Formanek, \emph{Braid group representations of low degree}, {Proc. London Math Soc.}, 73 (3), 279-322, (1996).

\bibitem{123} L. Kauffman, \emph{Virtual knot theory}, {Eur. J. Comb.}, 20(7), 663-690, (1999).

\bibitem{12345} L. Kauffman, \emph{A survey of virtual knot theory}, Knots in Hellas–98 (Delphi), Ser. Knots Everything, 24, World Sci. Publ., River Edge, NJ, 143-202, (2000).

\bibitem{888} L. Kauffman and S. Lambropoulou, \emph{Virtual braids}, {Fundamenta Mathematicae}, 184, 159-186, (2004).

\bibitem{Kram2002} D. Krammer, \emph{Braid groups are linear}, Annals Math., 155(1), (2002), 131-156.

\bibitem{Law90} R. Lawrence, \emph{Homological representations of the Hecke algebra}, {Comm. Math. Phys.}, 135(1), (1990), 141–191.

\bibitem{long} D. Long and M. Paton, \emph{The Burau representation of the braid group $B_n$ is not faithful for $n\geq 6$}, {Topology}, 32, (1992), 439-447.

\bibitem{37} T. Mayassi and M. Nasser, \emph{Classification of homogeneous local representations of the singular braid monoid}, {Arab. J. Math.}, 15, 307-329 (2026).

\bibitem{17} Y. Mikhalchishina, \emph{Local representations of braid groups}, {Sib. Math. J.}, 54(4), 666-678, (2013).

\bibitem{Moo} J. Moody, \emph{The Burau representation of the braid group $B_n$ is not faithful for large $n$}, {Bull. Amer. Math.Soc.}, 25, 379-384, (1991).

\bibitem{55} M. Nasser, \emph{Necessary and sufficient conditions for the irreducibility of a linear representation of the braid group $B_n$}, {Arab. J. Math.}, 13, 333-339, (2024).

\bibitem{40} M. Nasser, \textit{Local extensions and $\Phi$-type extensions of some local representations of the braid group $B_n$ to the singular braid monoid $SM_n$}, Vietnam Journal of Mathematics, 1-12, (2025).

\bibitem{Sys} I. Sysoeva, \emph{Irreducible Representations of Braid Group $B_n$ of dimension $n+1$}, {Journal of group Theory}, 24, 39-78, (2020).

\end{thebibliography}
\end{document}